\documentclass{amsart}
\usepackage{graphicx}

\newtheorem{theorem}{Theorem}[section]
\newtheorem{proposition}{Proposition}[section]
\newtheorem{lemma}[theorem]{Lemma}

\theoremstyle{definition}

\theoremstyle{remark}

\numberwithin{equation}{section}

\newcommand{\dde}[2]{{\dfrac {\partial {#1}}{\partial {#2}}}}
\newcommand{\dt}{\dfrac{d~}{dt}}

\newcommand{\R}{\mathbb R}

\newcommand{\Z}{\mathbb Z}

\newcommand{\imm}{\mathrm{i}}              
\newcommand{\e}{\mathrm{e}}                
\newcommand{\de}{\mathop{}\!\mathrm{d}}  
\newcommand{\pa}{\mathop{}\!\partial}

\newcommand{\fcr}{\mathcal X}
\newcommand{\Cal}[1]{{\mathcal {#1}}}

\begin{document}

\title[On the complete phase synchronization]{On the complete phase synchronization for the Kuramoto model
in the mean-field limit}


\author{D. Benedetto}
\email{benedetto@mat.uniroma1.it}

\author{E. Caglioti}
\email{caglioti@mat.uniroma1.it}

\author{U. Montemagno}
\email{montemagno@mat.uniroma1.it}
\address{Dipartimento di Matematica, {\it Sapienza} Universit\`a di Roma - 
P.zle A. Moro, 5, 00185 Roma, Italy}

\subjclass{92D25, 74A25, 76N10}

\keywords{Kuramoto model, complete synchronization, coupled oscillators}

\date{\today}

\dedicatory{}

\begin{abstract}
We study the Kuramoto model for coupled oscillators.  
For the case of identical natural frequencies, we  give a new
proof of the complete frequency syncronization for all initial data;
extending this result to the continuous version of the model, 
we manage to prove the complete phase
synchronization for any non-atomic measure-valued initial datum,
We also discuss the relation between the boundedness of the entropy 
and the convergence to an incoherent state, for the case of non
identical natural frequencies.
\end{abstract}

\maketitle
\noindent

\section{Introduction}
\label{sez:introduzione}

The Kuramoto model is a mean-field model of coupled oscillators,
which exhibits spontaneous synchronization in a certain 
range of parameters (see \cite{kuramoto},\cite{strogatz2000}).
The equations for the phases of the oscillators 
are
\begin{equation}\label{kuram}
 \dot{\vartheta}_i(t)=\omega_i-\frac{K}{N}\sum_{j=1}^{N}\sin({\vartheta}_i(t)-{\vartheta}_j(t)),\quad i=1,\dots N,
\end{equation}
where the phases $\vartheta_i$ 
can be considered in the one-dimensional torus $\Cal T$, i.e.
defined mod $2\pi$.
The parameters 
$\omega_i$ are the 'natural frequencies' of the  oscillators and
$K>0$ is the coupling intensity.
It can be useful to
represent the system \eqref{kuram} in the unitary circle
in the complex plane
by considering $N$ particles with position $\e^{\imm \vartheta_i(t)}$.
The center of mass  is in the point
\begin{equation}\label{deford}
 R(t)e^{\imm\varphi(t)}=\frac{1}{N}\sum_{j=1}^Ne^{\imm\vartheta_j(t)},
\end{equation}
where $0 \le R(t) \le 1$ and $\varphi(t)$ is well defined only if
$R(t)>0$. Using this definition the system 
\eqref{kuram} can be rewritten as
\begin{equation}\label{kuram1}
 \dot{\vartheta}_i(t)=\omega_i-
KR(t)\sin(\vartheta_i(t)-\varphi(t)),\quad i=1,\dots, N
\end{equation}
as follows from easy calculations.
The interaction term here becomes an attraction term 
towards the center of mass, 
and the intensity of the attraction is moduled by $R(t)$ which grows when the particles get closer. 

As shown in
\cite{strogatz2000}, \cite{bonilla},  \cite{HHK}, 
for large values of $K$ this system exhibits {\it complete
frequency synchronization}, 
i.e. for all $i$ and $j$
$$\vartheta_i-\vartheta_j \to \text{const.}, \ \ R \to \text{const. in }(0,1]$$
as $t\to +\infty$
and all the phases asymptotically rotate with the mean frequency
$$\omega = \frac 1N \sum_{i=1}^N \omega_i.$$
In the case of {\it identical oscillators}, i.e. if 
$\omega_i = \omega$ for all $i$, it is  possible
the {\it complete phase synchronization},
i.e. that $\vartheta_i-\vartheta_j\to 0$ and $R\to 1$.

For $K=0$, 
eq.s \eqref{kuram1} describe a free motion on the $N-$dimensional
torus ({\it incoherent state}). 
For intermediate values the asymptotic behaviour is
more complex: some 
oscillators can synchronize, while others move following 
their natural frequencies.
The asymptotic behaviour of $R$ is strictly related 
to the synchronization, so it is the ``order parameter''
for this phenomenon.

The complete synchronization has been studied with various
methods (see \cite{HHK} and \cite{dong} and references therein).
 In \cite{HHK} the authors consider also 
the case of identical oscillators: they prove the exponential 
convergence to a complete
phase synchronized state  
for initial data supported in an arc of $\Cal T$  
with length less then $\pi$.
This bound is optimal: $\vartheta_1(t) \equiv 0$ and 
$\vartheta_2(t) \equiv \pi$ is a stationary solution 
of \eqref{kuram} 
if $\omega_1 = \omega_2 = 0$.
In \cite{dong} the authors prove the complete 
frequency synchronization of identical oscillators for any initial datum.

In this work, 
in section \ref{sez:discreto}, we preliminary prove
 the complete 
frequency synchronization of identical oscillators with a different 
method with respect to \cite{dong},  
and we also analyze the case in which we obtain 
complete phase synchronization, showing that it is, in effect, 
the ``typical'' behaviour of the system of identical 
oscillators (see Theorem \ref{km1}).

Our method works also for the model obtained in the limit 
of infinitely many identical oscillators, in which 
the unknown is a measure $\rho(t,\vartheta)$ on $\Cal T$:
in section \ref{sez:continuo} we prove the complete 
frequency synchronization for any initial datum $\rho_0$,
and the complete phase synchronization if $\rho_0$ is non-atomic, 
i.e. if it gives zero measure to the points (see Theorem \ref{kkt}).
In this sense we extend a results of \cite{carrillo},
in which the authors prove the complete phase synchronization
if $\rho_0$ has support in a half circle.

In section \ref{sez:kneq} we analyze the 
case of non-identical oscillators, with the partial results of
Proposition \ref{convnoid}. Finally, we discuss the relation 
between the boundedness of the entropy and the convergence to 
an incoherent state.

\section{$N$ identical oscillators}
\label{sez:discreto}

Without loss of generality, 
we can choose $\omega_i=\omega = 0$ for all $i$, because we can subtract
$\omega t$ to the phases. Moreover, scaling the time, we can set $K=1$.
The system now reads as
\begin{equation}\label{kuramid}
 \dot{\vartheta}_i(t)=-\frac1{N}
\sum_{j=1}^{N}\sin({\vartheta}_i(t)-{\vartheta}_j(t)) = 
-R(t) \sin (\vartheta_i(t) - \varphi(t))
\end{equation}
where $R$ e $\varphi$ are defined in \eqref{deford}.
Eq.s \eqref{kuram} are a gradient system, namely
\begin{equation}
\label{gradiente}
\dot \vartheta_i = \dde U{\vartheta_i},
\ \ \text{ where }\ \ U(\vartheta_1,\dots \vartheta_N) = 
\frac 1{2N}\sum_{h,j=1}^N\cos (\vartheta_h - \vartheta_j)
\end{equation}
Note that $U$ is a function of $R$:
$$U = \frac N2 R^2$$
as follows from the following identities obtained by \eqref{deford}:
\begin{equation}
\label{cossin}
R = \frac 1N \sum_{i=1}^N \cos (\vartheta_i - \phi)\ \ \text{ and } \ \ 
\frac 1N \sum_{i=1}^N \sin (\vartheta_i - \phi)= 0.
\end{equation}
The system is invariant under translation, and the mean phase is conserved,
as follows by direct computation:
\begin{equation}
\label{traslazioni}
\frac{1}{N}\sum_{j=1}^N\vartheta_j(t)=\frac{1}{N}\sum_{j=1}^N\vartheta_j(0)
\end{equation}
Without loss of generality we assume the r.h.s. to be zero.

\vskip.3cm
It is simple to find the stationary solutions
of the system, remembering that we are in the framework
of zero mean frequency.
\begin{proposition}\label{stationary}
$\{\vartheta_i^*\}_{i=1}^N$ is a stationary solution of
\eqref{kuramid} iff one of the following properties holds
\begin{itemize}
\item[(1)] $R\equiv 0$ 
\item[(2)] $\left\{\vartheta_i^*\right\}_{i=1}^N$ is of type $(N-k,k)$, 
  that is there exists $\varphi^*$ such that:\\
  $\vartheta_i^* = \varphi^*$ mod $2\pi$, for  $i \in I$,\\ 
  $\vartheta_i^*=\varphi^*+ \pi$ mod $2\pi$, for $i \in I^c$\\
  where $I\subseteq \{1,\dots , N\}$ is a subset of indices with $|I|=k>N/2$.
\end{itemize}
\end{proposition}
\noindent
The first case corresponds to an incoherent state:
the center of mass is in the origin and 
$\varphi$ is undefined; these solutions 
form translational invariant submanifolds of the torus 
$\Cal T^N$ of dimension $N-1$.
In the second case $R= 1-2k/N$ and
$\varphi = \varphi^*$. If $k=0$, the solution is a complete phase 
synchronized state, while if $k\ge 1$ the solution is only 
a complete frequency synchronized state.

\vskip.3cm
It is easy to prove that the absolute maximum of the function $U$ 
is achieved by complete synchronization states, i.e. stationary solutions
of the type $(N,0)$, which are the only stable solutions
of the system. Removing the translational invariance by
fixing the mean phase,  all the
critical point of $U$ are isolated but for the minima
which corresponds to $R=0$.
Moreover, it can be proved that
the value of $\vartheta_i$ are bounded in time (see \cite{dong}).
The gradient structure \eqref{gradiente}  
allows the authors  in \cite{dong} to prove the
complete frequency synchronization of the system for any
initial data. As a consequence, it is easy to 
prove that the solutions must converge
to a complete phase synchronized state, up to a 0-measure set of initial data,
which corresponds to the unstable stationary solutions and their
stable manifolds.

This kind of arguments cannot be used
in the limit $N\to +\infty$, so we use a different method,
which is based  on the analysis of the asymptotic behaviour of $R$
and $\varphi$.

\begin{proposition}\label{rmon}
If $\{\vartheta_i(t)\}_{i=1}^N$ is not a stationary solution, then 
\begin{itemize}
\item[(1)] $\dot{R}(t)>0, \quad \forall\, t>0$
\item[(2)] $R(t)\xrightarrow[t \rightarrow \infty]{}R^*\in (0,1]$
\item[(3)] $\varphi(t)$ is well defined $\forall\, t>0$
\end{itemize}
\end{proposition}
\begin{proof}
Deriving in $t$ eq. \eqref{deford}, after some 
manipulations, we obtain:
\begin{equation}
\label{rdot}
\dot{R}(t)=\left[\frac{1}{N}\sum_{j=1}^N\sin^2(\vartheta_j(t)-\varphi(t))\right]R(t)
\end{equation}
from which the thesis.
\end{proof}
\vskip.3cm

In the sequel we use the following calculus lemma.
\begin{lemma}\label{lemma1}
Let $f$ be a $C^1$ function $f:[0,+\infty)\to \R$, with  $|f'(t)|\le C$.
If the integral $\int_0^\infty f(s)\de s$ exists finite, then
$f(t)\xrightarrow[t \rightarrow \infty]{}0$.
\end{lemma}

Now we have all the ingredients to prove the main result 
of this section.
\begin{theorem}\label{km1}
If $\vartheta_i(t), \,\, i=1,\dots, N$ is not a stationary solution,  
then it converges to a completely 
frequency synchronized state of type $(N-k,k)$.

Moreover, 
if  $\vartheta_i(0)\not=\vartheta_j(0)\mod 2\pi \,\,\text{when}\,\, i\not=j$, 
the solution converges to a stationary solution of type 
$(N,0)$ or $(N-1,1)$.
\end{theorem}
\proof
Since $R(t)\to R^*$, as stated in Proposition \ref{rmon},  
$\dot R$ verifies the hypothesis of Lemma 
$\ref{lemma1}$, then, using eq. \eqref{rdot}
\begin{equation}
 \frac{1}{N}\sum_{j=1}^N\sin^2(\vartheta_j(t)-\varphi(t))\to 0
\end{equation}
therefore
\begin{equation}
 \sin(\vartheta_j(t)-\varphi(t))\to 0,\quad  j=1, \dots N.
\end{equation}
Since $\sin x$ has isolated zeros
\begin{equation}
 \vartheta_j(t)-\varphi(t)\to k_j\pi
\end{equation}
for some $k_j\in \Z$.
Using that, from the assumption \eqref{traslazioni}, the mean phase  is zero 
\begin{equation}
 \varphi(t)=-\frac{1}{N}\sum_{j=1}^N\left[\vartheta_j(t)- \varphi(t)\right]
\to -\frac{1}{N}\sum_{j=1}^Nk_j\pi=:\varphi^*.
\end{equation}
\noindent
Finally, $\vartheta_j(t)$ converges $\forall \, j=1,\dots, N$
\begin{equation}
 \vartheta_j(t)=\vartheta_j(t)-\varphi(t)+\varphi(t)\to k_j\pi+\varphi^*.
\end{equation}

In order to prove the second part of the theorem, we assume $k\ge 2$. 
Then there exist $i,j$ such that
\begin{equation}
 \vartheta_i(t)\to \varphi^*+(2k_i+1)\pi, \quad \vartheta_j(t)\to \varphi^*+(2k_j+1)\pi, \quad k_i, k_j \in \mathbb{Z}
\end{equation}
We can write
\begin{equation}\label{van}
 \vartheta_h(t)-\varphi(t)=\xi_h(t)+(2k_h+1)\pi, \quad h=i,j
\end{equation}
where  $\xi_h\to 0$ and
$\dot{\xi}_h(t)=R(t)\sin(\xi_h(t))-\dot{\varphi}(t)$, as follows
 from \eqref{kuramid}. Then, 
using that $x\sin(x/2)\geq x^2/\pi$ when $x\in[-\pi,\pi]$
and that $\xi_i$, $\xi_j$ go to zero:
\begin{align}\label{expl}
 \dt [\xi_i(t)-\xi_j(t)]^2&=2(\xi_i(t)-\xi_j(t))[R(t)\sin(\xi_i(t))-R(t)\sin(\xi_j(t))]=\\ \nonumber
 &=4(\xi_i(t)-\xi_j(t))R(t)\sin\left(\frac{\xi_i(t)-\xi_j(t)}{2}\right)\cos\left(\frac{\xi_i(t)+\xi_j(t)}{2}\right)\geq\\  \nonumber
 &\geq C(\xi_i(t)-\xi_j(t))^2,
\end{align}
\noindent
which contradicts the fact that $\xi_i,\,\xi_j\to 0$.
\endproof
\noindent
It is not possible to exclude that the limit 
point is a stationary solution of type $(N-1,1)$,
in fact its stable manifold is clearly non-empty,
as can be easily verified. 
For istance, consider the case of three oscillators
with $\vartheta_1(t)=-\vartheta_2(t) = \delta(t)$, 
$\vartheta_{3}(t)\equiv\pi$,
where $\delta(t)$ satisfies the equation
$$\dot \delta = \frac {2}{3} \sin \delta \left(
\frac 1{2} - \cos \delta\right)$$
The asymptotic behaviour depends on the initial datum 
$\delta(0)=\delta_0$:

$$\lim_{t\to + \infty}
\delta(t) = \left\{ \begin{aligned}
0 & \text{ if } \delta_0 \in [0,\pi/3)\\
\pi/3 & \text{ if }\delta_0 = \pi/3\\
\pi & \text{ if }\delta_0 \in (\pi/3,\pi]
\end{aligned}\right.
$$
In the first case the solution tends to a stationary solution of 
type $(2,1)$, which is a complete frequency synchronized state;
in the second case, the system is in an incoherent state; 
in the last case we have complete phase synchronization.

\section{The kinetic model for identical oscillators}
\label{sez:continuo}

We now consider the dynamics induced by \eqref{kuramid},
in the limit $N\to +\infty$
for a density of phases $\rho(t,\vartheta)$ defined on $\Cal T$
(see \cite{carrillo})
The equation for $\rho$ is a conservation law
of current $v$ depending non locally on $\rho$:
\begin{equation}\label{kkuram}
\begin{cases}\partial_t \rho(t,\vartheta)
+\pa_\vartheta(v(t,\vartheta)\rho(t,\vartheta))=0\\
v(t,\vartheta) = - \int_{\Cal T}  \sin(\vartheta -
\vartheta') \rho(t,\vartheta') \de \vartheta'\\
\end{cases}
\end{equation}
This equation has a weak form for which
existence and uniqueness results for measure valued solution 
 have been proved in \cite{lancellotti} (see also \cite{carrillo}):
\begin{equation}
\label{kuraidweak}
\left\{\begin{aligned}
&\dot{\Theta}(t,\vartheta)=-R(t) \sin(\Theta(t,\vartheta) - \varphi(t)),\ \ 
\text{with }{\Theta}(0,\vartheta) = \vartheta\\[3pt]
&R(t)e^{\imm\varphi(t)}=\int_{\Cal T} e^{\imm\vartheta}\rho(t,\vartheta)d\vartheta
\\[3pt]
&\int_{\Cal T}
h(\vartheta) \rho(t,\vartheta)  \de \vartheta = 
\int_{\Cal T} 
h(\Theta(t,\vartheta)) \rho_0(\vartheta)  \de \vartheta  
\end{aligned}\right.
\end{equation}
where the measure $\rho_0(\vartheta)$ is the initial datum
and $h$ is any regular $2\pi$-periodic observable.

The order parameters verify the identities
\begin{equation}
\label{rphicont}
R(t) = \int_{\Cal T} \cos(\eta - \varphi(t)) \rho(t,\eta) \de \eta, \ \ 
\text{ and } \int_{\Cal T} \sin(\eta - \varphi(t)) \rho(t,\eta) \de \eta =0,
\end{equation} 
and their time derivatives  are
\begin{equation}
\label{orderpar}
\begin{aligned}
\dot{R}(t)&=R(t)\int_{\Cal T}\sin^2(\eta-\varphi(t))\rho(t,\eta)d\eta \\
\dot{\varphi}(t)&=-\int_{\Cal T}\sin(\eta-\varphi(t))\cos(\eta-
\varphi(t))\rho(t,\eta)d\eta \end{aligned}
\end{equation}
Also in this case the mean phase is constant in time
\begin{align}
\int_{[-\pi,\pi)} \Theta(t,\vartheta) \rho_0(\vartheta)d\vartheta=
\int_{[-\pi,\pi)}  \vartheta \rho_0(\vartheta)d\vartheta,
\end{align}
because its  time derivative is zero as follows from 
\eqref{rphicont}.
Note that $\vartheta$ is not an observable 
on $\Cal T$, so $\int_{[-\pi,\pi)} \vartheta \rho(t,\vartheta)d\vartheta$ 
is different from 
$\int_{[-\pi,\pi)}  \vartheta \rho_0(\vartheta)d\vartheta$.

The stationary solutions of \eqref{kkuram} are a generalization 
of the ones relative to the discrete system.
\begin{proposition}
\label{propo:uguali-cin}
$\rho(t,\vartheta)\equiv \rho^*(\vartheta)$ is a stationary solution
of \eqref{kkuram} iff it verifies one of the following identities:
\begin{itemize}
\item[(1)] $R=0$ 
\item[(2)] $\rho^*(\vartheta)$ is of type $(c_1,c_2)$, that is
$\rho^*(\vartheta)
= c_1\delta(\vartheta - \varphi^*)+c_2
\delta(\vartheta - \varphi^*-\pi)$, where $c_1>c_2\geq 0$
and $c_1+c_2=1$.
\end{itemize}
\end{proposition}
\proof
Let $\rho^*(\vartheta)$ be a stationary solution and
$R^*$ and $\varphi^*$ be the corresponding order parameters. The current
is 
$$v(t,\theta) = v(\vartheta) = -R^*\sin (\vartheta -\varphi^*)$$
and $v\rho^*$ must be constant, then $R^*=0$ or 
$\rho^*$ is supported where $\sin(\vartheta-\varphi^*)$ is zero.
\endproof

The discrete model is a particular case of the kinetic model
\eqref{kkuram}, but for the second one the proof of the convergence 
is a little more difficult, so we have to adapt 
our argument.
We start proving that there exists the asymptotic value
of $\varphi(t)$.

\begin{proposition}
$\varphi(t)$ converges when $t$ goes to infinity.
\end{proposition}
\proof
As in the discrete case, 
for a non stationary solution $R(t)\to R^*\in (0,1]$, as follows 
from the first of \eqref{orderpar}. Applying 
Lemma \ref{lemma1}
\begin{align}\label{sinzero}
\int_{\Cal T}\sin^2(\eta-\varphi(t))\rho(t,\eta)d\eta
\xrightarrow[t\longrightarrow +\infty]{}0;
\end{align}
and  from the second of \eqref{orderpar} we have
\begin{align}
|\dot{\varphi}(t)|&\leq\int_{\Cal T}|\sin(\eta-\varphi(t))\cos(\eta-\varphi(t))|\rho(t,\eta)d\eta\leq\\
&\leq\left[ \int_{\Cal T}\sin^2(\eta-\varphi(t))\rho(t,\eta)d\eta\right]^{\frac{1}{2}}\xrightarrow[t\longrightarrow +\infty]{}0,\nonumber
\end{align}
which implies that $\int_0^\infty \ddot{\varphi}(s)ds$ is finite.\\
Using again Lemma \ref{lemma1} and doing some calculations 
\begin{align}
\ddot{\varphi}(t)&=\dt \left[-\int_{\Cal T}\sin(\eta-\varphi(t))\cos(\eta-\varphi(t))\rho(t,\eta)d\eta\right]=\\
&=\int_{\Cal T} \cos (2(\eta-\varphi(t)))[\dot{\varphi}(t)+R(t)\sin(\eta-\varphi(t))]\rho(t,\eta)d\eta=\nonumber \\
&=\int_{\Cal T} (1-2\sin^2(\eta-\varphi(t)))[\dot{\varphi}(t)+R(t)\sin(\eta-\varphi(t))]\rho(t,\eta)=\nonumber\\
&=\dot{\varphi}(t)-2\int_{\Cal T}\sin^2(\eta-\varphi(t))[\dot{\varphi}(t)+R(t)\sin(\eta-\varphi(t))]\rho(t,\eta)\nonumber,
\end{align}
where, in the last identity, we used the second of eq.s \eqref{rphicont}.
The second term is  bounded by 
$$4\int_{\Cal T}\sin^2(\eta-\varphi(t))\rho(t,\eta)d\eta$$
which is summable in  $t\in [0,+\infty)$, then
\begin{align}
\phi^* := 
\lim_{t\to \infty}\varphi(t)=\phi(0)+
\int_0^\infty \dot{\phi}(s)ds \,\,\mbox{exists 
finite}.
\end{align}
\endproof
\noindent
Using this result we can prove the convergence of 
the characteristics $\Theta(t,\vartheta)$.
\begin{proposition}{}
\label{convtheta}
There exist $\alpha \in \Cal T$ such that
$$\lim_{t\to +\infty} 
\Theta(t,\vartheta ) = \begin{cases}
\phi^* \text{ for } \vartheta \in \Cal T\backslash \{\alpha\}\\
\phi^*+\pi \text{ for } \vartheta  = \alpha 
\end{cases}
$$
(these identities must by intended mod $2\pi$).
\end{proposition}
\begin{proof}
Consider $\Cal T$ as $\phi^*+[-\pi,\pi]$, and, for $n\ge 1$,
define the partition of $\Cal T$
$$\begin{aligned}
&A^0_{n} = \phi^*+[-1/n,1/n],\ \ A^\pi_{n}= \phi^*+[\pi-1/n,\pi] \cup
[-\pi,-\pi+1/n],\\
&B^+_n = \phi^*+(1/n,\pi-1/n),\ \ B^-_n = \phi^*+(-\pi+1/n,-1/n)
\end{aligned}$$
Since $R(t)\to R^*$ and $\varphi(t)\to \varphi^*$, 
there exists an increasing diverging sequence $t_n$ such that, for $t\ge t_n$
\begin{equation}\label{disB}
R(t)|\sin(\vartheta -\varphi(t))| \ge \frac{R^*}2 
|\sin(\vartheta - \varphi^*)|\ 
\text{ for } \vartheta \in B^+_n\cup B^-_n
\end{equation}
The subsets $G_{n} = A^0_{n}\cup B^+_n \cup B^-_n$
is invariant, in the sense that if for $\bar t \ge t_n$ 
$\Theta(\bar t,\vartheta) \in G_n$ then
$\Theta(t,\vartheta) \in G_n$ for all $t\ge \bar t$.
Note that also $A^0_n$ is invariant.
As a consequence,
if $n> m$
$$\Theta(-t_{n},A^\pi_{n})  \subseteq \Theta(-t_{m},A^\pi_{m})
\ \ \text{ and }\ \ 
\Theta(-t_{m},A^0_{m})  \subseteq \Theta(-t_{n},A^0_{n})
$$
Since  $\Theta(-t_{n},A^\pi_{n})$ are arcs of $\Cal T$,
it is well defined the arc
$$[\alpha_1,\alpha_2] = \bigcap_{n\ge 1} \Theta(-t_{n},A^\pi_{n})$$
By definition, if $\vartheta \notin [\alpha_1,\alpha_2]$
then $\Theta(t,\vartheta) \in G_n$ for all $t\ge t_n$, for all $n$. 
But for a finite $\tau$, independent on $n$, 
$\Theta(t_n+\tau,\vartheta) \in A^0_n$.
Using the invariance of $A^0_n$, we obtain that
$\Theta(t,\vartheta) \to \varphi^*$.

If $\vartheta \in [\alpha_1,\alpha_2]$, 
$\Theta(t_n,\vartheta)\in A^\pi_n$ for all $n$.
Suppose now that there exists $\bar t>t_n$ such that
$\Theta(\bar t,\vartheta)\notin A^\pi_n$. Then for the invariance 
of $G_n$, for all $m$ such that $t_{m} \ge \bar t$, 
$\Theta(t_{m},\vartheta)\in G_n\subset G_m$, and then
$\vartheta \in \Theta(-t_{m},G_{m})$ 
in contrast with the hypothesis on $\vartheta\in [\alpha_1,\alpha_2]$.
We conclude that 
$\Theta(t,\vartheta)\to \varphi^*+\pi$.

Finally, we can repeat the same argument of the proof of the 
second part of Theorem \ref{km1}, 
showing that, since $\Theta(t,\alpha_i) \to \varphi^*+\pi$, $i=1,2$,
asymptotically
$$\dt [\Theta(t,\alpha_1) - \Theta(t,\alpha_2)]^2 \ge  
C [\Theta(t,\alpha_1) - \Theta(t,\alpha_2)]^2.$$
then $\alpha_1 = \alpha_2 = \alpha$.
\end{proof}

\begin{theorem}\label{kkt}
If $\rho(\vartheta,t)$ is not a stationary solution then 
\begin{equation}
\rho(t,\vartheta) \xrightarrow[t \rightarrow \infty]{weak-*}\rho^*(\vartheta).
\end{equation}
\noindent
where $\rho^*(\vartheta)$ is a stationary solution of type $(c_1,c_2)$.

Moreover, if $\rho_0(\vartheta)$ is non-atomic, then 
$\rho^*=\delta(\vartheta-
\varphi^*)$, i.e. is a complete phase synchronized state.
\end{theorem}
\begin{proof}
Let be $h$ a regular periodic observable.
Using Proposition \ref{convtheta}
$$\int_{\Cal T} h(\vartheta) \rho(t,\vartheta)\de \vartheta
= \int_{\Cal T} h(\Theta(t,\vartheta)) \rho_0(\vartheta)\de \vartheta 
\to c_1 h(\phi^*) + c_2 h(\phi^*+\pi)$$
where
$$c_1 = \int_{\Cal T\backslash \{\alpha\}}  \rho_0(\vartheta)\de \vartheta,\ \
c_2 = 1-c_1$$
and $c_2$ is the measure that $\rho_0$ gives to the point $\alpha$,
which is zero if $\rho_0(\vartheta)$ gives zero measure to the points.
\end{proof}

\section{Some considerations on the kinetic model for non identical oscillators}

The following equations describe the dynamic of infinitely many non
identical oscillators in the kinetic limit.
\label{sez:kneq}

\begin{equation}\label{noneq}
\begin{cases}
\partial_t f(t,\vartheta,\omega)+ \partial_\vartheta(v(t,\vartheta,\omega)f(t,\vartheta,\omega))=0,\\
v(t,\vartheta,\omega)=\omega -K\int_{\mathcal{T}\times \R} \sin(\vartheta-\vartheta')f(t,\vartheta',\omega')d\vartheta'd\omega'\\
\end{cases},
\end{equation}
where $f(t,\vartheta,\omega)$ is a positive $2\pi$-periodic function
in $\vartheta$, which represents the probability density of oscillators
with phase $\vartheta$ and frequency $\omega$.  The marginal
$\rho(t,\vartheta)=\int_\R f(t,\vartheta,\omega)d\omega$ is the
probability density of the phases.
The distribution of the natural frequencies 
is $g(\omega ) = \int_{\Cal T}  f(t,\vartheta, \omega)d\vartheta$,
which is a conserved quantity.

A reference for existence and uniqueness results for this equation is
still \cite{lancellotti} where the kinetic model \eqref{noneq} is
rigorously derived by doing the $N\to \infty$ limit of \eqref{kuram}.
A weak 
formulation of \eqref{noneq} can be given in terms of the 
characteristics $\Theta(t,\vartheta,\omega)$:
\begin{equation}\label{weakne}
\left\{ 
\begin{aligned}
&\dot{\Theta}(t,\vartheta,\omega)=\omega-KR(t)\sin({\Theta}(t,\vartheta,\omega)-\varphi(t)),\,\, \mbox{with} \,\,{\Theta}(0,\vartheta,\omega)=\vartheta\\
&R(t)\e^{\imm\varphi(t)}=
\int_{\Cal T\times \R} \e^{\mathbf{i}\vartheta}f(t,\vartheta,\omega)d\vartheta d\omega \\
&\int_{\Cal T \times \R} f(t,\vartheta,\omega)h(\vartheta,\omega)d\vartheta d\omega
  =\int_{\Cal T \times \R} f_0(\eta,\omega)h(\Theta(t,\eta,\omega),\omega)d\eta d\omega\\
\end{aligned}\right.
\end{equation}
where $h$ is any regular function of $(\vartheta,\varphi)\in 
\mathcal{T}\times\mathbb{R}$.
 Without loss of generality we can assume
\begin{equation}
\langle\omega\rangle=\int_{\R} \omega g(\omega)d\omega=0\quad \langle\vartheta\rangle=\int_{[-\pi,\pi)}\vartheta \rho_0(\vartheta)d\vartheta=0,
\end{equation}
By the previous assumption it following that
\begin{align}
&\int_{[-\pi,\pi)\times \R} \Theta(t,\vartheta,\omega)f_0(\vartheta,\omega) d\vartheta d\omega=0.
\end{align}

As shown in \cite{bonilla}, when $g$ has compact support and $K$ is
sufficiently large, there exist stationary solutions 
$f^*$, which are in some sense the analogous of the two delta solutions
for the case of identical oscillators described in Proposition  
\ref{propo:uguali-cin}.
Imposing the current $v=\omega
-KR\sin(\theta-\varphi)$ to be zero we obtain
\begin{equation}
\begin{aligned}\label{sss}
&f^*(\vartheta,\omega)=g^+(\omega)\delta(\vartheta-\vartheta^+(\omega))+g^-(\omega)\delta(\vartheta-\vartheta^-(\omega))\\
&\vartheta^+(\omega)=\varphi^*+\arcsin\left(\frac{\omega}{KR}\right),\,\,\vartheta^-(\omega)=\pi +\varphi^*-\arcsin\left(\frac{\omega}{KR}\right)\\
&g^+,g^-\geq 0,\,\, g^+ + g^-=g
\end{aligned}
\end{equation}
with $R$ satisfying the following equation of 
self consistency
(which
has solutions for $K$ large enough)
\begin{equation}
KR^2=\int_{\R}\sqrt{(KR)^2-\omega^2}\Big[g^+(\omega)-g^-(\omega)\Big]d\omega.\\
\end{equation}
Taking the marginal of $f^*$ the particle density $\rho^*$ is
\begin{align}
 \rho^*(\vartheta)&=KR|\cos(\vartheta-\varphi^*)|g^+(KR\sin(\vartheta-\varphi^*))\fcr_{|\vartheta-\varphi^*|<\frac{\pi}{2}}+\\
 &+KR|\cos(\vartheta-\varphi^*)|g^-(KR\sin(\vartheta-\varphi^*))\fcr_{|\vartheta -(\pi +\varphi^*)|<\frac{\pi}{2}}.\nonumber
\end{align}
Particular relevance have the stable solutions (see \cite{carrillo})
which are the ones with $g^+(\omega) = g(\omega)$ and $g^-(\omega) =0$:
\begin{align}\label{staz}
f^*(\vartheta,\omega)=g(\omega)\delta\left(\vartheta-\vartheta^+(\omega)\right)
\end{align}
where $R\in(0,1]$ is the largest solution of 
\begin{equation}\label{selfc}
KR^2=\int_\R\sqrt{(KR)^2-\omega^2}g(\omega)d\omega.\\
\end{equation}
To the authors' knowledge 
the best result of
 convergence to an equilibrium of this kind is in \cite{carrillo},
where complete frequency synchronization is proved for initial phases lying
in a compact subset of $(-\frac{\pi}{2},\frac{\pi}{2})$, 
although the expression of equilibrium density 
$\rho^*(\vartheta)=KR\cos(\vartheta-\varphi^*)g(KR\sin(\vartheta-\varphi^*))$
is not explicitly written in the paper.

In the case of non identical oscillators, the order parameter $R(t)$ is
no more increasing in general:
\begin{equation}
\dot{R}(t)=KR(t) \int_{ \Cal T}
\sin^2(\eta-\varphi(t))\rho(t,\eta)d\eta-
\int_{\Cal T\times \R} \omega\sin(\eta-\varphi(t))f(t,\eta,\omega)d\eta d\omega
\end{equation}
\begin{align}
R(t)\dot{\varphi}(t)&=-KR(t)
\int_{\Cal T}\sin(\eta-\varphi(t))\cos(\eta-\varphi(t))\rho(t,\eta)d\eta\nonumber\\
 &+\int_{\Cal T\times \R}
\omega\cos(\eta-\varphi(t))f(t,\eta,\omega)d\eta d\omega
\end{align}
then we can not extend the convergence result of the previous
section to this case.
Nevertheless, we can characterize the 
possible limits of the solution, excluding also in 
this case the ``two delta 
solutions'' as generic asymptotic behaviour.
\begin{proposition}
\label{convnoid}
Suppose that $f_0(\vartheta,\omega)\de \vartheta$ is non-atomic 
for any $\omega$. 
If, as $t\to +\infty$,  
$R(t)\to R^*>0$ with $\mathrm{supp}\ g \subset [-KR^*,KR^*]$, and 
$\varphi(t) \to \varphi^*$,
than $f$ converges weakly to $f^*$ given by \eqref{staz} and 
$R^*$ solves \eqref{selfc}.
\end{proposition}
The proof follows as in Proposition \ref{convtheta}
and Theorem \ref{kkt}: we first can prove the convergence of 
$\Theta(t,\vartheta,\omega)$ to $\vartheta^\pm(\omega)$ mod $2\pi$, 
for $|\omega| < KR^*$,
than we show there exists only one value of $\vartheta\in \Cal T$ such that
$\Theta(t,\vartheta,\omega) \to \vartheta^-(\omega)$ mod $2\pi$, 
finally we prove the weak convergence of $f$ using the 
convergence of the characteristics.

Note that eq. \eqref{selfc} can have two solutions (see \cite{lyapunov},
\cite{mirollo}),
then Proposition \ref{convnoid} does not assure the convergence 
to the stable stationary solution.

\vskip.3cm

The asymptotic behaviour in the case of non identical oscillators
can be complex, even 
if the system is still of gradient type (in a different space).
The functional is
\begin{equation}\label{ener}
\mathcal{H}_{f}(t)=\int_{[-\pi,\pi)\times \R} 
\Theta(t,\vartheta,\omega) \omega f_0(\vartheta,\omega)d\vartheta d
\omega+K\frac{R^2(t)}{2}.
\end{equation}
which is non decreasing along the solutions.
In contrast with the case of identical oscillators, 
this functional is not well defined on the function on $\Cal T\times R$, 
and it is unbounded.

There is another functional with a monotone behaviour, 
related to the entropy.
\begin{proposition}
If $\Theta(t,\vartheta, \omega)$ is solution of \eqref{weakne} then
\begin{equation}
\label{jacobiano}
\frac{d}{dt}\int_{{\Cal T}\times \R} 
\ln\left(\frac{\partial}{\partial \vartheta}\Theta(t,\vartheta, \omega)\right)f_0(\vartheta,\omega)d\vartheta d\omega=-KR^2(t)
\end{equation}
\end{proposition}
\proof
Th r.h.s of \eqref{jacobiano} is
\begin{equation}
\begin{aligned}
&\int_{{\Cal T}\times \R} 
\left[\frac{\partial\Theta}{\partial \vartheta}\right]^{-1}
\frac{d}{dt}\frac{\partial\Theta}{\partial \vartheta} 
f_0(\vartheta,\omega)d\vartheta d\omega \\
 &=-\int_{{\Cal T}\times \R}  KR(t)\cos(\Theta(t,\vartheta,\omega)-\varphi(t))f_0(\vartheta,\omega)d\vartheta d\omega=-KR^2(t)
\end{aligned}\end{equation}
\endproof
This proposition makes sense for any initial data of \eqref{noneq}, 
and shows the tendency of the system to  shrink the solution in the 
$\vartheta$ variabile.
If $f_0$ is absolute continuous w.r.t. the Lebesgue measure, 
this functional can be rewritten, modulus a constant
that explodes if $f$ become singular, in a more eloquent way.
\begin{proposition}
The entropy is non decreasing along the solutions of \eqref{noneq}
\begin{equation}
\frac{d}{dt}\int_{\Cal T\times \R} f(t,\vartheta,\omega)\ln(f(t,\vartheta,\omega))d\vartheta d\omega=KR^2(t).
\end{equation}
\end{proposition}
\noindent
In the hypothesis of the results presented in \cite{carrillo} 
and in that of Proposition \ref{convnoid}, 
while $f$ approaches $f^*$ the entropy grows to
infinity. 

If the entropy does not diverge, $R\to 0$,
so the system behaves as an incoherent state.
More precisely we can prove the following two propositions.

\begin{proposition}
If the entropy does not diverge, then 
the functional  $\mathcal{H}_{f}$ asymptotically grows as in the 
case of incoherent states: the limit
\begin{equation}
\lim _{t \to \infty}\left[\int_{[-\pi,\pi)\times \R} \Theta(t,\vartheta,\omega)\omega f_0(\vartheta, \omega)d\vartheta d\omega - \left(\int_{[-\pi,\pi)\times \R} \omega^2 f_0(\vartheta, \omega)d\vartheta d\omega\right)t \right]
\end{equation}
is finite. In other words, the functional  $\mathcal{H}_{f}$ grows as in the
case of free flows ($K=0$).
\end{proposition}
\proof
Firstly we write the derivative
\begin{align}
&\frac{d}{dt}\left[\int_{[-\pi,\pi)\times \R} (\Theta(t,\vartheta,\omega) -\omega t)\omega f_0(\vartheta, \omega)d\vartheta d\omega\right]=\\
&=-KR(t) \int_{\Cal T \times \R}\omega \sin(\vartheta-\varphi(t)) f(t,\vartheta, \omega)d\vartheta d\omega\nonumber.
\end{align}
Now we write the derivative of $\frac{R^2(t)}{2}$
\begin{align}\label{tre}
\frac{d}{dt}\frac{R^2(t)}{2}=&\left[\int_{\Cal T}\sin^2(\eta-\varphi(t))\rho(t,\eta)d\eta\right]KR^2(t)+\\
&-\left[\int_{\Cal T\times \R}\omega\sin(\eta-\varphi(t))f(t,\eta,\omega)d\eta d\omega\right]R(t)\nonumber
\end{align}
Integrating the last identity between 0 and t we get
\begin{align}\label{quat}
&\frac{R^2(t)}{2}-\frac{R^2(0)}{2}=\int_0^t\left[\int_{\Cal T}\sin^2(\eta-\varphi(s))\rho(s,\eta)d\eta\right]KR^2(s)ds+\\
&-\int_0^t\left[\int_{\Cal T\times \R} \omega\sin(\eta-\varphi(s))f(s,\eta,\omega)d\eta d\omega\right]R(s)ds\nonumber
\end{align}
using both (\ref{tre}) and (\ref{quat}) we get
 \begin{align}
& \int_{[-\pi,\pi)\times \R} (\Theta(t,\vartheta,\omega) -\omega t)\omega f_0(\vartheta, \omega)d\vartheta d\omega= \int_{[-\pi,\pi)\times \R} \vartheta\omega f_0(\vartheta, \omega)d\vartheta d\omega+\\
& +\frac{R^2(t)}{2}-\frac{R^2(0)}{2}-\int_0^t\left[\int_{\Cal T}\sin^2(\eta-\varphi(s))\rho(s,\eta)d\eta\right]KR^2(s)ds\nonumber
  \end{align}
we conclude because $R^2(t)$ is integrable and $R(t) \xrightarrow[t \rightarrow \infty]{}0$ (for lemma \ref{lemma1}).
\endproof
\begin{proposition}
If the entropy does not diverge, 
the only possible limit points of $f(t,\vartheta,\omega)$ are incoherent states 
\end{proposition}
\begin{proof} \ \\
1st step: if the entropy does not diverge, then 
$$\int_{\Cal T\times \R} e^{\mathbf{i}\vartheta}\omega^k f(t,\vartheta,\omega)d\vartheta d\omega \to 0,\forall k\in \mathbb{N}.$$
The proof is done by induction, the first step is the fact that if the entropy does not diverge, then $R(t)$ vanishes by Lemma \ref{lemma1}.
Now we do the inductive step
\begin{equation}
\int_{\Cal T \times \R} e^{\mathbf{i}\vartheta}\omega^k f(t,\vartheta,\omega)d\vartheta d\omega \to 0 \Rightarrow \int_{\Cal T \times \R} e^{\mathbf{i}\vartheta}\omega^{k+1} f(t,\vartheta,\omega)d\vartheta d\omega \to 0.
\end{equation}
We write the derivative of the l.h.s.
\begin{align}
&\frac{d}{dt}\left[\int_{\Cal T \times \R} e^{\mathbf{i}\vartheta}\omega^k f(t,\vartheta,\omega)d\vartheta d\omega\right]=\\
&=\int_{\Cal T \times \R} \mathbf{i}e^{\mathbf{i}\vartheta}\omega^k\left[\omega-KR(t)\sin(\vartheta-\varphi(t)) \right] f(t,\vartheta,\omega)d\vartheta d\omega.
\nonumber
\end{align}
This quantity satisfies the hypothesis of the Lemma \ref{lemma1}, so it goes to zero, which implies 
\begin{equation}
 \int_{\Cal T \times \R} e^{\mathbf{i}\vartheta}\omega^{k+1} f(t,\vartheta,\omega)d\vartheta d\omega \to 0.
\end{equation}
2nd step: the limit points are incoherent states.

Let's call $\bar {f}(\vartheta,\omega)$ 
a limit point of $f(t,\vartheta,\omega)$, then we have

\begin{equation}\label{inc}
 \int e^{\mathbf{i}\vartheta}\omega^{k} 
{\bar f}(\vartheta,\omega)d\vartheta d\omega = 0,\quad \forall k\in \mathbb{N}.
\end{equation}
The solution of the Equation \eqref{noneq} with  $\bar f$ as initial datum 
is ${\bar f}(\vartheta -\omega t, \omega)$, 
in fact $R(t)$ generated by this density is zero:
\begin{align}
R(t)&=\int_{\Cal T \times \R} e^{\mathbf{i}\vartheta} {\bar f}(\vartheta-\omega t,\omega)d\vartheta d\omega = \int_{\Cal T \times \R} e^{\mathbf{i}(\eta+\omega t)} {\bar f}(\eta,\omega)d\eta d\omega=\\
&=\int_{\Cal T \times \R} e^{\mathbf{i}\eta} e^{\mathbf{i}\omega t} {\bar f}(\eta,\omega)d\eta d\omega=\int_{\Cal T \times \R} e^{\mathbf{i}\eta}\sum_{k=0}^\infty \frac{(\mathbf{i}\omega)^k}{k!}{\bar f}(\eta,\omega)d\eta d\omega=\\\nonumber
&=\sum_{k=0}^\infty\frac{\mathbf{i}^k}{k!}\int_{\Cal T \times \R} e^{\mathbf{i}\eta} \omega^k{\bar f}(\eta,\omega)d\eta d\omega=0.\nonumber
\end{align}
\end{proof}
\noindent 
Note that a density $\bar f(\vartheta,\omega)$ is an incoherent state
iff its first Fourier coefficient in $\vartheta$ is zero for any $\omega$.
\bibliographystyle{plain}
%
\bibliography{kuramoto2014}

\begin{thebibliography}{1}

\bibitem{bonilla}
J.A. Acebr\'on, L.L. Bonilla, C.J. P\'erez~Vicente, F.~Ritort, and R.~Spigler.
\newblock The kuramoto model: A simple paradigm for synchronization phenomena.
\newblock {\em Rev. Mod. Phys.}, 77(137), 2005.

\bibitem{carrillo}
J.A. Carrillo, Y.-P. Choi, S.-Y. Ha, M.-J. Kang, and Y.~Kim.
\newblock Contractivity of transport distances for the kinetic kuramoto
  equation.
\newblock {\em Journal of Statistical Physics}, 156(2):395--415, 2014.

\bibitem{dong}
J.-G. Dong and X.~Xue.
\newblock Synchronization analysis of {K}uramoto oscillators.
\newblock {\em Commun. Math. Sci.}, 11(2):465--480, 2013.

\bibitem{HHK}
S.-Y. Ha, T.~Ha, and J.-H. Kim.
\newblock On the complete synchronization of the {K}uramoto phase model.
\newblock {\em Phys. D}, 239(17):1692--1700, 2010.

\bibitem{kuramoto}
Y.~Kuramoto.
\newblock Self-entrainment of a population of coupled non-linear oscillators.
\newblock In Huzihiro Araki, editor, {\em International Symposium on
  Mathematical Problems in Theoretical Physics}, volume~39 of {\em Lecture
  Notes in Physics}, pages 420--422. Springer Berlin Heidelberg, 1975.

\bibitem{lancellotti}
C.~Lancellotti.
\newblock On the {V}lasov limit for systems of nonlinearly coupled oscillators
  without noise.
\newblock {\em Transport Theory Statist. Phys.}, 34(7):523--535, 2005.

\bibitem{mirollo}
R.E. Mirollo and S.H. Strogatz.
\newblock The spectrum of the locked state for the kuramoto model of coupled
  oscillators.
\newblock {\em Physica D: Nonlinear Phenomena}, 205(1–-4):249--266, 2005.

\bibitem{strogatz2000}
S.H. Strogatz.
\newblock From {K}uramoto to {C}rawford: exploring the onset of synchronization
  in populations of coupled oscillators.
\newblock {\em Phys. D}, 143(1-4):1--20, 2000.
\newblock Bifurcations, patterns and symmetry.

\bibitem{lyapunov}
J.L. van Hemmen and W.F. Wreszinski.
\newblock Lyapunov function for the kuramoto model of nonlinearly coupled
  oscillators.
\newblock {\em Journal of Statistical Physics}, 72(1-2):145--166, 1993.

\end{thebibliography}
\end{document}